\documentclass{amsart}

\newtheorem{theorem}{Theorem}[section]

\newtheorem{cor}{Corollary}[section]
\theoremstyle{definition}

\newtheorem{ex}[theorem]{Example}

\theoremstyle{remark}

\numberwithin{equation}{section}

\begin{document}

\title{Remarks on Jordan derivations over matrix algebras}

\author{Arindam Ghosh}
\address{Department of Mathematics, Indian Institute of Technology Patna, Patna-801106}
\curraddr{}
\email{E-mail: arindam.pma14@iitp.ac.in}
\thanks{}

\author{Om Prakash$^{\star}$}
\address{Department of Mathematics, Indian Institute of Technology Patna, Patna-801106}
\curraddr{}
\email{om@iitp.ac.in}
\thanks{* Corresponding author}

\subjclass[2010]{47B47, 47L35}

\keywords{Jordan Derivations, Derivations, Upper Triangular Matrix Algebra, Full Matrix Algebra}

\date{}

\dedicatory{}

\begin{abstract}
Let $C$ be a commutative ring with unity. In this article, we show that every Jordan derivation over an upper triangular matrix algebra $\mathcal{T}_n(C)$ is an inner derivation. Further, we extend the result for Jordan derivation on full matrix algebra $\mathcal{M}_n(C)$.
\end{abstract}

\maketitle

\section{Introduction}
Throughout this article $C$ denotes a commutative ring with unity, unless otherwise stated and $T$ a $C$-algebra. Recall that a map $D:T\rightarrow T$ is called a \emph{Jordan derivation} if it is $C$-linear and $D(a^2)=D(a)a+aD(a)$, for all $a\in T$. This is said to be a \emph{derivation} if $D(ab)=D(a)b+aD(b)$, for all $a, b\in T$ and an \emph{antiderivation} if $D(ab)=D(b)a+bD(a)$, for all $a, b\in T$.  Derivations and antiderivations are the trivial examples of Jordan derivations. But not every Jordan derivation is a derivation (Example of \cite{E}). A derivation $D: T\rightarrow T$ is said to be an \emph{inner derivation} if there exists $a_0\in T$ such that $D(a)=a_0a-aa_0$, for all $a\in T$.

The study of Jordan derivation was initiated by Herstein in 1957. In \cite{D}, he had shown that there is no proper Jordan derivation for a prime ring of characteristic not 2. In 1975, Cusack extended Herstein's result in \cite{D1}. Later on, in 1988 Bre$\check{s}$ar \cite{F} proved that every Jordan derivation from a 2-torsion free semiprime ring into itself is a derivation. Note that a ring $R$ is \emph{2-torsion free} if for $a \in R$ such that $2a=0_R$, then $a=0_R$. The problem whether every Jordan derivation of a ring or algebra into itself is a derivation was discussed by many mathematicians in \cite{ben,F2,F4}.

Now, suppose $A$ and $B$ are unital algebras over $C$, and $M$ a unital $(A,B)$-bimodule which is faithful as a left $A$-module and as a right $B$-module. Then the $C$-algebra
\[\text{Tri}(A,B,M)=\left \{\left( \begin{array}{ccc}
a & m \\
0 & b \end{array} \right)\vert \enspace a \in A,\, b \in B,\, m \in M \right \} \]
under the usual matrix operations is said to be a triangular algebra. In last few years Jordan derivation over triangular algebras has invited attention of many mathematicians. In 2005, Benkovi$\check{c}$ \cite{G} proved that every Jordan derivation from an upper triangular matrix algebra $\mathcal{A}$ into its arbitrary bimodule $\mathcal{M}$ is the sum of a derivation and an antiderivation, where $\mathcal{M}$ is 2-torsion free. In 2006, Zhang and Yu \cite{S} proved that every Jordan derivation from $U$ into itself is a derivation, where $U=\text{Tri}(A,B,M)$ is a triangular algebra and $C$ is 2-torsion free. Note that above result is not true if $C$ is not $2$-torsion free. In this connection an example of a Jordan derivation over $\text{Tri}(A,B,M)$ is given which is not a derivation. The construction of the example is same as in Example 8 by Cheung in \cite{che}, but in other perspective.
\begin{ex}
Let $C=\mathbb{Z}_2$, $A=B=\left \{\left( \begin{array}{ccc}
a & b \\
0 & a \end{array} \right)\vert \enspace a,~b  \in \mathbb{Z}_2 \right \}$ and $M=\mathcal{T}_2({\mathbb{Z}_2})$, where $\mathbb{Z}_2=\{0,1\}$, a field with two elements and $\mathcal{T}_2({\mathbb{Z}_2})$ is the algebra of $2\times 2$ matrices over $\mathbb{Z}_2$. Note that $\mathbb{Z}_2$ is not a $2$-torsion free ring. Define $D:\text{Tri}(A,B,M)\rightarrow \text{Tri}(A,B,M)$ as
\begin{align*}
& D(X)=a_{34}e_{12}+a_{24}e_{13}+a_{13}e_{24}+a_{12}e_{34}~\text{for all}~X\in \text{Tri}(A,B,M),\\
& \text{where}~X=a_{11}(e_{11}+e_{22})+a_{33}(e_{33}+e_{44})+a_{12}e_{12}+a_{13}e_{13}+a_{14}e_{14}+a_{24}e_{24}\\
&+a_{34}e_{34},~a_{ij}\in \mathbb{Z}_2.
\end{align*}
Here $e_{ij}$ represents the $n\times n$ matrix with $1$ at $(i,j)$th position and $0$ elsewhere. Then $D$ is a Jordan derivation over $\text{Tri}(A,B,M)$. Let $X=e_{11}+e_{13}+e_{14}+e_{22}+e_{33}+e_{34}+e_{44}$ and $Y=e_{13}+e_{14}+e_{33}+e_{34}+e_{44}$. In this case, $D(XY)=0\neq e_{12}=D(X)Y+XD(Y)$. Therefore, $D$ is not a derivation over $\text{Tri}(A,B,M)$.
\end{ex}

Motivated by above, in Section 2, we prove that a Jordan derivation from the algebra of upper triangular matrices $\mathcal{T}_n(C)$ into itself is a derivation, without assuming $C$ to be $2$-torsion free (Theorem \ref{3.1}). This is towards the furtherance of Theorem 2.1 of \cite{S}. Moreover, it has been proved that derivation on $\mathcal{T}_n(C)$ is an inner derivation (Theorem \ref{3.3}). In 2007 \cite{R}, Ghosseiri proved that if $R$ is a $2$-torsion free ring, $n \geq 2$, and $D$ is a Jordan derivation on the upper triangular matrix ring $\mathcal{T}_n(R)$, then $D$ is a derivation. As a consequence of Theorem \ref{3.1}, a remark is provided on Ghosseiri's result that Jordan derivation over $\mathcal{T}_n(F)$, where $F=\{0,1\}$ is not $2$- torsion free, is a derivation (Corollary \ref{3.2}).

In 2009 \cite{ali}, Alizadeh proved that if $A$ is a unital associative ring and $M$ is a $2$-torsion free $A$-bimodule, then every Jordan derivation from $\mathcal{M}_n(A)$ into $\mathcal{M}_n(M)$ is a derivation. In Section 3, we prove that there is no such proper Jordan derivations on the full matrix algebra $\mathcal{M}_n(C)$, without assuming $C$ is $2$-torsion free (Theorem \ref{4.1}). Further, it is proved that every derivation on $\mathcal{M}_n(C)$ is an inner derivation (Theorem \ref{4.2}).

Before stating the main results, we have the following:\\

Let $T$ be an algebra, $0_T$ and $1_T$ represent zero and identity of $T$ respectively. Similarly $0$ and $1$ represent the zero and identity of $C$ respectively. $e_{ij}$ denotes the square matrix $(e_{ij})_{n\times n}$ with $1$ at $(i,j)$th position and $0$ elsewhere. Let $D: T \rightarrow T$ be a Jordan derivation. Then $D(x^2)=D(x)x+xD(x)$ for all $x\in T$. Replacing $x$ by $x+y$, we have $D((x+y)^2)=D(x+y)(x+y)+(x+y)D(x+y)$, which implies for all $x,y \in T$,\begin{equation}\label{eq:1}
D(xy+yx)=D(x)y+xD(y)+D(y)x+yD(x).
\end{equation}
Since $D$ is additive,
\begin{equation}
\label{eq:2}
D(0_T)=0_T.
\end{equation}
Also, $D(1_T)=D(1_T^2)=D(1_T)1_T+1_TD(1_T)$, therefore
\begin{equation}
\label{eq:3}
D(1_T)=0_T.
\end{equation}

\section{Jordan Derivation on $\mathcal{T}_n(C)$}
In this section, we discuss about Jordan derivation which is an inner derivation over upper triangular matrix algebras and towards this we prove the following:

\begin{theorem}
\label{3.1}
Let $C$ be a commutative ring with unity and $\mathcal{T}_n(C)$ be an algebra of $n\times n$ upper triangular matrices over $C$. Then every Jordan derivation on $\mathcal{T}_n(C)$, $n \geq 2$ into itself is a derivation.
\end{theorem}

\begin{proof}
Let $T=\mathcal{T}_n(C)$ and $D$ be a Jordan derivation from $T$ into itself. Let
\begin{eqnarray}
\label{eq:3a}
D(e_{ii})=\sum\limits_{1 \leq j \leq k \leq n}a_{jk}^{(ii)}e_{jk}~, ~\text{where} ~a_{jk}^{(ii)} \in C.
\end{eqnarray}
Since $D$ is a Jordan derivation, $D(e_{ii})=D(e_{ii}^2)=D(e_{ii})e_{ii}+e_{ii}D(e_{ii})$. Now, by using \eqref{eq:3a} and equating the coefficients from both sides, we obtain
\begin{equation}
\label{eq:4}
\begin{aligned}
D(e_{ii})=a_{1i}^{(ii)}e_{1i}+a_{2i}^{(ii)}e_{2i}+a_{3i}^{(ii)}e_{3i}+\cdots+a_{i-1,i}^{(ii)}e_{i-1,i}+a_{i,i+1}^{(ii)}e_{i,i+1}+\cdots+a_{in}^{(ii)}e_{in}\\
\text{for all} ~i\in \{1,2,3,\cdots,n\}.
\end{aligned}
\end{equation}

 Also, from \eqref{eq:3}, we have $D(1_T)=0_T$. Therefore, by using \eqref{eq:4}, we have $\sum\limits_{1 \leq i < j \leq n}(a_{ij}^{(ii)}+a_{ij}^{(jj)})e_{ij}=0_T$. Hence,
\begin{eqnarray}
\label{eq:5}
a_{ij}^{(ii)}+a_{ij}^{(jj)}=0~, ~\text{for} ~1 \leq i<j \leq n.
\end{eqnarray}

In order to find $D(e_{ij})$, when $1 \leq i < j\leq n$, let $D(e_{ij})=\sum\limits_{1 \leq k \leq l \leq n}a_{kl}^{(ij)}e_{kl}$, for $a_{kl}^{(ij)} \in C$ and $1 \leq i < j\leq n$. Now, by applying \eqref{eq:1} on $e_{ij}=e_{ii}e_{ij}+e_{ij}e_{ii}$, we have $D(e_{ij})=D(e_{ii})e_{ij}+e_{ii}D(e_{ij})+D(e_{ij})e_{ii}+e_{ij}D(e_{ii})$. Also, by \eqref{eq:4}, we get,
\begin{equation}
\label{eq:5a}
\begin{aligned}
D(e_{ij}) &= a_{1i}^{(ii)}e_{1j} + a_{2i}^{(ii)}e_{2j} + a_{3i}^{(ii)}e_{3j} + \cdots + a_{i-1,i}^{(ii)}e_{i-1,j} \\
&+ a_{1i}^{(ij)}e_{1i} + a_{2i}^{(ij)}e_{2i} + \cdots + a_{i-1,i}^{(ij)}e_{i-1,i} + a_{i,i+1}^{(ij)}e_{i,i+1} + \cdots + a_{in}^{(ij)}e_{in}.
\end{aligned}
\end{equation}

\vspace{0.2cm}
Since $e_{ij}=e_{ij}e_{jj}+e_{jj}e_{ij}$. Therefore, by using \eqref{eq:1} and putting the value of $D(e_{ij})$ from \eqref{eq:5a} and $D(e_{jj})$ from \eqref{eq:4}, we have
\begin{equation}
\label{eq:6a}
D(e_{ij})=a_{1i}^{(ii)}e_{1j} + a_{2i}^{(ii)}e_{2j} + a_{3i}^{(ii)}e_{3j} + \cdots + a_{i-1,i}^{(ii)}e_{i-1,j} + a_{ij}^{(ij)}e_{ij} + a_{j,j+1}^{(jj)}e_{i,j+1} + \cdots + a_{jn}^{(jj)}e_{in}.
\end{equation}

Now from the assumption of $D(e_{ij})$ and \eqref{eq:6a},
\begin{equation}
\label{eq:6}
D(e_{ij})=a_{1i}^{(ii)}e_{1j} + a_{2i}^{(ii)}e_{2j} + a_{3i}^{(ii)}e_{3j} + \cdots + a_{i-1,i}^{(ii)}e_{i-1,j} + a_{ij}^{(ij)}e_{ij} + a_{i,j+1}^{(ij)}e_{i,j+1} + \cdots + a_{in}^{(ij)}e_{in}.
\end{equation}

Now, we establish
\begin{equation}
\label{eq:7}
D(e_{ij}e_{kl})=D(e_{ij})e_{kl}+e_{ij}D(e_{kl})
\end{equation}
for all $i,j,k,l\in \{1,2,3,\cdots,n\}$. Since \eqref{eq:7} is equivalent to
\begin{equation}
\label{eq:8}
D(e_{kl}e_{ij})=D(e_{kl})e_{ij}+e_{kl}D(e_{ij}).
\end{equation}
Therefore, proof of \eqref{eq:7} is sufficient to justify \eqref{eq:8} and vice-versa. During the proof, we frequently use \eqref{eq:4} and \eqref{eq:6}.

\underline{Case 1:} Let $i\neq j$. Without loss of generality, let $i>j$. Since from \eqref{eq:4}, we have
\begin{align*}
D(e_{ii})e_{jj}+e_{ii}D(e_{jj})=0_T.
\end{align*}
Therefore, $D(e_{ii}e_{jj})=D(e_{ii})e_{jj}+e_{ii}D(e_{jj})$.

\underline{Case 2:} Let $j<k$. In this case we want to establish $D(e_{ii}e_{jk})=D(e_{ii})e_{jk}+e_{ii}D(e_{jk})$ for all $i,j,k\in \{1,2,3,\cdots,n\}$.

If $i<j$, then by using \eqref{eq:5},
\begin{align*}
D(e_{ii})e_{jk}+e_{ii}D(e_{jk})=(a_{ij}^{(ii)}+a_{ij}^{(jj)})e_{ik}=0_T.
\end{align*}

If $i=j$, then
\begin{align*}
D(e_{ik})e_{ii}+e_{ik}D(e_{ii})=0_T.
\end{align*}

If $i>j$, then
\begin{align*}
D(e_{ii})e_{jk}+e_{ii}D(e_{jk})=0_T.
\end{align*}

\underline{Case 3:} Let $i<j$ and $k<l$. Now, we prove $D(e_{ij}e_{kl})=D(e_{ij})e_{kl}+e_{ij}D(e_{kl})$ for all $i,j,k,l\in \{1,2,3,\cdots,n\}$.

For $j<k$,
\begin{align*}
D(e_{kl})e_{ij}+e_{kl}D(e_{ij})=0_T.
\end{align*}

If $j=k$, then
\begin{align*}
D(e_{kl})e_{ik}+e_{kl}D(e_{ik})=0_T.
\end{align*}

Also, for $j>k$,
\begin{align*}
D(e_{ij})e_{kl}+e_{ij}D(e_{kl})=0_T.
\end{align*}

 Now to prove $D$ is a derivation on $T$, let $A$, $B$ $\in T$, where $A=\sum\limits_{1 \leq i \leq j \leq n}A_{ij}e_{ij}$ and $B=\sum\limits_{1 \leq k \leq l \leq n}B_{kl}e_{kl}$ for some $A_{ij}, B_{kl} \in C$. Since $D$ is a Jordan derivation on $T$ and a derivation on $e_{ij}$'s. Hence
\begin{align*}
D(A)B+AD(B) &=  \sum\limits_{1 \leq i \leq j \leq n} \sum\limits_{1 \leq k \leq l \leq n}A_{ij}B_{kl}[D(e_{ij})e_{kl}+e_{ij}D(e_{kl})] \\
            &=  \sum\limits_{1 \leq i \leq j \leq n} \sum\limits_{1 \leq k \leq l \leq n}A_{ij}B_{kl}D(e_{ij}e_{kl}) \\
            &=  D(\sum\limits_{1 \leq i \leq l \leq n}C_{il}e_{il})\\
            &=  D(AB),
\end{align*}
where $C_{il}=\sum_{j=i}^{l}A_{ij}B_{jl}$. Thus, $D$ is a derivation on $T$.
\end{proof}

Now as a corollary, we describe Jordan derivation on $\mathcal{T}_n(F)$, where $n \geq 2$ is a positive integer and $\mathcal{T}_n(F)$ is considered as a ring. In this case, we relax the linearity condition of the map.

\begin{cor}
\label{3.2}
If $F$ is a field with two elements, then every Jordan derivation on $\mathcal{T}_n(F)$, $n \geq 2$, into itself is a derivation.
\end{cor}

\begin{proof}
Let $D$ is a Jordan derivation on $\mathcal{T}_n(F)$. Since $F=\{0,1\}$, $D$ is $F$-linear. Hence $D$ is a derivation by Theorem \ref{3.1}.
\end{proof}

\begin{theorem}
\label{3.3}
Every derivation over $\mathcal{T}_n(C),~n\geq 2,$ is an inner derivation.
\end{theorem}

\begin{proof}
Let $D$ be a derivation on $\mathcal{T}_n(C)$. Since every derivation is a Jordan derivation, all the identities in the proof of Theorem \ref{3.1} hold.

Let $i<j<k$. Since $e_{ik}=e_{ij}e_{jk}$ and $D$ is a derivation on $\mathcal{T}_n(C)$, we have
$D(e_{ik})=D(e_{ij})e_{jk}+e_{ij}D(e_{jk})$. By \eqref{eq:6}, equating the coefficient of $e_{ik}$ from both sides,
\begin{equation}
\label{8b}
a_{ik}^{(ik)}=a_{ij}^{(ij)}+a_{jk}^{(jk)}.
\end{equation}

Let $X=\sum\limits_{1 \leq j \leq k \leq n}x_{jk}e_{jk}$, where $x_{jk} \in C$. Now, by using \eqref{eq:4}, \eqref{eq:5}, \eqref{eq:6} and \eqref{8b}, $D(X)=BX-XB$, where
\begin{align*}
B=\sum_{l=2}^{n}(-a_{1l}^{(1l)})e_{ll}+\sum\limits_{1 \leq i < j \leq n}a_{ij}^{(jj)}e_{ij}.
\end{align*}
So, $D$ is an inner derivation.

\end{proof}

As a consequence of Theorem \ref{3.1} and \ref{3.3}, we have the following:

\begin{cor}
\label{3.4}
Every Jordan derivation over $\mathcal{T}_n(C),~n\geq 2,$ is an inner derivation.
\end{cor}

Let $D$ be a Jordan derivation on $\mathcal{T}_n(C)$. The question is whether there exists a unique $B\in \mathcal{T}_n(C)$ so that $D(X)=BX-XB$, for all $X\in \mathcal{T}_n(C)$. The answer is given by the following example.
\begin{ex}
Define $D:\mathcal{T}_2(\mathbb{Z}) \rightarrow \mathcal{T}_2(\mathbb{Z})$ by \\
\[D \begin{pmatrix}
x_{11}&x_{12}\\
0&x_{22}
\end{pmatrix}= \begin{pmatrix}
0&x_{12}\\
0&0
\end{pmatrix},\\
~\text{where}~x_{ij}\in \mathbb{Z},~\text{the ring of integers}. \]
By easy computation, $D$ is a Jordan derivation. Also, by Corollary \ref{3.4}, $D$ is inner. Note that $D(X)=BX-XB$ for $B = e_{11}$ or $B = -e_{22}$. Hence, we have more than one choices in this case for $B$.
\end{ex}


\section{ Jordan Derivations on $\mathcal{M}_n(C)$}
Now, we state and prove our main theorem of Jordan derivation on full matrix algebra.

\begin{theorem}
\label{4.1}
Let $C$ be a commutative ring with unity and $\mathcal{M}_n(C)$ be the full matrix algebra over $C$. Then every Jordan derivation on $\mathcal{M}_n(C)$, $n \geq 2$ into itself is a derivation.
\end{theorem}

\begin{proof}
Let $T=\mathcal{M}_n(C)$, $D$ be a Jordan derivation from $T$ into itself. Let
\begin{eqnarray}
\label{eq:8a}
D(e_{ii})=\sum_{k=1}^{n}\sum_{l=1}^{n}a_{kl}^{(ii)}e_{kl}~, ~\text{for} ~a_{kl}^{(ii)} \in C.
\end{eqnarray}
Since $D$ is a Jordan derivation, $D(e_{ii})=D(e_{ii}^2)=D(e_{ii})e_{ii}+e_{ii}D(e_{ii})$. Now, by using \eqref{eq:8a} and equating the coefficients from both sides, we obtain,
\begin{equation}
\label{eq:9}
\begin{aligned}
D(e_{ii}) &= a_{i1}^{(ii)}e_{i1}+ \cdots +a_{i,i-1}^{(ii)}e_{i,i-1}+a_{i,i+1}^{(ii)}e_{i,i+1}+ \cdots +a_{in}^{(ii)}e_{in} \\
      &+ a_{1i}^{(ii)}e_{1i}+ \cdots +a_{i-1,i}^{(ii)}e_{i-1,i}+a_{i+1,i}^{(ii)}e_{i+1,i}+ \cdots +a_{ni}^{(ii)}e_{ni} \\
      &\text{for all} ~i\in \{1,2,3,\cdots,n\}.
\end{aligned}
\end{equation}

Also,
\begin{equation}
\label{eq:10}
\begin{aligned}
D(1_T)=0_T & \implies \mathop {\sum_{k=1}^{n}\sum_{l=1}^{n}}_{k\neq l}(a_{kl}^{(kk)}+a_{kl}^{(ll)})e_{kl}=0_T \\
& \implies a_{kl}^{(kk)}+a_{kl}^{(ll)}=0
\end{aligned}
\end{equation}
for all $k,l$ with $k\neq l$, by using \eqref{eq:3} and \eqref{eq:9}.

In order to find $D(e_{ij})$ for $i\neq j$, let
\begin{eqnarray}
\label{eq:10a}
D(e_{ij})=\sum_{k=1}^{n}\sum_{l=1}^{n}a_{kl}^{(ij)}e_{kl}~, ~\text{for} ~a_{kl}^{(ij)} \in C ~\text{and} ~i\neq j.
\end{eqnarray}

From \eqref{eq:2}, we have the identity $D(e_{ij})e_{ij}+e_{ij}D(e_{ij})=D(e_{ij}^{2})=D(0_T)=0_T$. By equating the coefficients of $e_{ij}$ and $e_{ii}$, we obtain
\begin{equation}
\label{eq:11a}
a_{ii}^{(ij)}+a_{jj}^{(ij)}=0
\end{equation}
and
\begin{equation}
\label{eq:12}
a_{ji}^{(ij)}=0
\end{equation}
respectively. Now, applying \eqref{eq:1} on $e_{ij}=e_{ii}e_{ij}+e_{ij}e_{ii}$, we get $D(e_{ij})=D(e_{ii})e_{ij}+e_{ii}D(e_{ij})+D(e_{ij})e_{ii}+e_{ij}D(e_{ii})$. Using \eqref{eq:9} and \eqref{eq:10a},
\begin{equation}
\begin{aligned}
\label{eq:12b}
D(e_{ij}) &= a_{1i}^{(ii)}e_{1j}+ \cdots +a_{i-1,i}^{(ii)}e_{i-1,j}+a_{i+1,i}^{(ii)}e_{i+1,j}+ \cdots +a_{ni}^{(ii)}e_{nj} \\
& + a_{i1}^{(ij)}e_{i1}+ \cdots +a_{i,i-1}^{(ij)}e_{i,i-1}+a_{ii}^{(ij)}e_{ii}+a_{i,i+1}^{(ij)}e_{i,i+1}+ \cdots +a_{in}^{(ij)}e_{in} \\
& + a_{1i}^{(ij)}e_{1i}+ \cdots +a_{i-1,i}^{(ij)}e_{i-1,i}+a_{ii}^{(ij)}e_{ii}+a_{i+1,i}^{(ij)}e_{i+1,i}+ \cdots +a_{ni}^{(ij)}e_{ni} \\
&+a_{ji}^{(ii)}e_{ii}.
\end{aligned}
\end{equation}

Now, from \eqref{eq:10a} and \eqref{eq:12b}, equating the coefficient of $e_{ii}$,
\begin{equation}
\begin{aligned}
\label{eq:12c}
a_{ii}^{(ij)}=2a_{ii}^{(ij)}+a_{ji}^{(ii)}.
\end{aligned}
\end{equation}

Therefore, from \eqref{eq:12b} and \eqref{eq:12c},
\begin{equation}
\begin{aligned}
\label{eq:12a}
D(e_{ij}) &= a_{1i}^{(ii)}e_{1j}+ \cdots +a_{i-1,i}^{(ii)}e_{i-1,j}+a_{i+1,i}^{(ii)}e_{i+1,j}+ \cdots +a_{ni}^{(ii)}e_{nj} \\
& + a_{i1}^{(ij)}e_{i1}+ \cdots +a_{i,i-1}^{(ij)}e_{i,i-1}+a_{ii}^{(ij)}e_{ii}+a_{i,i+1}^{(ij)}e_{i,i+1}+ \cdots +a_{in}^{(ij)}e_{in} \\
& + a_{1i}^{(ij)}e_{1i}+ \cdots +a_{i-1,i}^{(ij)}e_{i-1,i}+a_{i+1,i}^{(ij)}e_{i+1,i}+ \cdots +a_{ni}^{(ij)}e_{ni}.
\end{aligned}
\end{equation}

Since $e_{ij}=e_{ij}e_{jj}+e_{jj}e_{ij}$. So by using \eqref{eq:1}, \eqref{eq:10} and \eqref{eq:12} and putting the values of $D(e_{ij})$ and $D(e_{jj})$ from \eqref{eq:12a} and \eqref{eq:9} respectively, we have
\begin{equation}
\label{eq:13a}
\begin{aligned}
D(e_{ij})  &= a_{1i}^{(ii)}e_{1j}+ \cdots +a_{i-1,i}^{(ii)}e_{i-1,j}+a_{i+1,i}^{(ii)}e_{i+1,j}+ \cdots +a_{ni}^{(ii)}e_{nj}+a_{ij}^{(ij)}e_{ij} \\
& + a_{j1}^{(jj)}e_{i1}+ \cdots +a_{j,j-1}^{(jj)}e_{i,j-1}+a_{j,j+1}^{(jj)}e_{i,j+1}+ \cdots +a_{jn}^{(jj)}e_{in}.
\end{aligned}
\end{equation}

Also, from \eqref{eq:10a} and \eqref{eq:13a},
\begin{equation}
\label{eq:13}
\begin{aligned}
D(e_{ij})  &= a_{1i}^{(ii)}e_{1j}+ \cdots +a_{i-1,i}^{(ii)}e_{i-1,j}+a_{i+1,i}^{(ii)}e_{i+1,j}+ \cdots +a_{ni}^{(ii)}e_{nj} \\
& + a_{i1}^{(ij)}e_{i1}+ \cdots +a_{i,i-1}^{(ij)}e_{i,i-1}+a_{ii}^{(ij)}e_{ii}+a_{i,i+1}^{(ij)}e_{i,i+1}+ \cdots +a_{in}^{(ij)}e_{in}
\end{aligned}
\end{equation}

and
\begin{equation}
\label{eq:14}
a_{il}^{(ij)}=a_{jl}^{(jj)}~, ~\text{for all} ~l\in \{1,2,\cdots,j-1,j+1,\cdots,n\}.
\end{equation}

Now, equating the coefficient of $e_{jj}$ from \eqref{eq:10a} and \eqref{eq:13},
\begin{equation}
\label{eq:14a}
a_{jj}^{(ij)}=a_{ji}^{(ii)}
\end{equation}

Again, from \eqref{eq:11a} and \eqref{eq:14a},
\begin{equation}
\label{eq:11}
a_{ii}^{(ij)}+a_{ji}^{(ii)}=0.
\end{equation}

Now, we establish
\begin{equation}
\label{eq:15}
D(e_{ij}e_{kl})=D(e_{ij})e_{kl}+e_{ij}D(e_{kl})
\end{equation}
for all $i,j,k,l\in \{1,2,3,\cdots,n\}$. Since \eqref{eq:15} is equivalent to
\begin{equation}
\label{eq:16}
D(e_{kl}e_{ij})=D(e_{kl})e_{ij}+e_{kl}D(e_{ij}).
\end{equation}
Therefore, proof of \eqref{eq:15} is sufficient to justify \eqref{eq:16} and vice-versa. During the proof, we frequently use \eqref{eq:9} and \eqref{eq:13}.

\underline{Case 1:} Let $i\neq j$. Since from \eqref{eq:9} and \eqref{eq:10}, we have
\begin{align*}
D(e_{ii})e_{jj}+e_{ii}D(e_{jj})=(a_{ij}^{(ii)}+a_{ij}^{(jj)})e_{ij}=0_T.
\end{align*}
Therefore, $D(e_{ii}e_{jj})=D(e_{ii})e_{jj}+e_{ii}D(e_{jj})$.

\underline{Case 2:} Let $j\neq k$. In this case we  establish $D(e_{ii}e_{jk})=D(e_{ii})e_{jk}+e_{ii}D(e_{jk})$ for all $i,j,k\in \{1,2,3,\cdots,n\}$.

If $i\neq j$, then by using \eqref{eq:9}, \eqref{eq:10} and \eqref{eq:13},
\begin{align*}
D(e_{ii})e_{jk}+e_{ii}D(e_{jk})=(a_{ij}^{(ii)}+a_{ij}^{(jj)})e_{ik}=0_T.
\end{align*}

If $i=j$, then by using \eqref{eq:11},
\begin{align*}
D(e_{ik})e_{ii}+e_{ik}D(e_{ii})=(a_{ii}^{(ik)}+a_{ki}^{(ii)})e_{ii}=0_T.
\end{align*}

\underline{Case 3:} Let $i\neq j$ and $k\neq l$. Now, our goal is to prove  $D(e_{ij}e_{kl})=D(e_{ij})e_{kl}+e_{ij}D(e_{kl})$ for all $i,j,k,l\in \{1,2,3,\cdots,n\}$.

Let $j\neq k$. Then by Case 2 and \eqref{eq:1},
\begin{equation}
\label{eq:17}
D(e_{ij}e_{kk}) = 0_T \implies D(e_{ij})e_{kk}+e_{ij}D(e_{kk}) = 0_T \implies a_{ik}^{(ij)}+a_{jk}^{(kk)} = 0.
\end{equation}

Also, by using \eqref{eq:17}, we have
\begin{align*}
D(e_{ij})e_{kl}+e_{ij}D(e_{kl})=(a_{ik}^{(ij)}+a_{jk}^{(kk)})e_{il}=0_T.
\end{align*}

Let $j=k$ and $i\neq l$. Replacing $(i,j,k)$ by $(j,l,i)$  in \eqref{eq:17}, we have
\begin{align*}
D(e_{jl})e_{ij}+e_{jl}D(e_{ij})=(a_{ji}^{(jl)}+a_{li}^{(ii)})e_{jj}=0_T.
\end{align*}

Let $j=k$ and $i=l$. Interchanging $i$ and $j$ in \eqref{eq:14}, we get
\begin{align*}
D(e_{ij})e_{ji}+e_{ij}D(e_{ji}) &= a_{1i}^{(ii)}e_{1i}+ \cdots +a_{i-1,i}^{(ii)}e_{i-1,i}+a_{i+1,i}^{(ii)}e_{i+1,i}+ \cdots +a_{ni}^{(ii)}e_{ni} \\
      &+ a_{j1}^{(ji)}e_{i1}+ \cdots +a_{j,i-1}^{(ji)}e_{i,i-1}+a_{j,i+1}^{(ji)}e_{i,i+1}+ \cdots +a_{jn}^{(ji)}e_{in} \\
      &= a_{1i}^{(ii)}e_{1i}+ \cdots +a_{i-1,i}^{(ii)}e_{i-1,i}+a_{i+1,i}^{(ii)}e_{i+1,i}+ \cdots +a_{ni}^{(ii)}e_{ni} \\
      &+ a_{i1}^{(ii)}e_{i1}+ \cdots +a_{i,i-1}^{(ii)}e_{i,i-1}+a_{i,i+1}^{(ii)}e_{i,i+1}+ \cdots +a_{in}^{(ii)}e_{in} \\
      &= D(e_{ii}).
\end{align*}

This proves our claim.
Finally, proof of $D(AB)=D(A)B+AD(B)$, where $A, B \in \mathcal{M}_{n}(C)$ is same as in Theorem \ref{3.1}. Thus, $D$ is a derivation on $\mathcal{M}_{n}(C)$.
\end{proof}


To find an example of Jordan derivation over $\mathcal{M}_{n}(C)$, we get only inner derivation. Towards this, we have the following:\\

\begin{theorem}
\label{4.2}
Let $C$ be a commutative ring with unity and $\mathcal{M}_n(C)$ be the algebra of all $n\times n$ matrices over $C$. Then every derivation of $\mathcal{M}_{n}(C),~n\geq 2$ is an inner derivation.
\end{theorem}

\begin{proof}
Let $D$ be a derivation on $\mathcal{M}_n(C)$. Since every derivation is a Jordan derivation, all the identities in proof of the Theorem \ref{4.1} hold.

Let $i\neq j\neq k$. Since $e_{ik}=e_{ij}e_{jk}$ and $D$ is a derivation on $\mathcal{M}_n(C)$, we have
$D(e_{ik})=D(e_{ij})e_{jk}+e_{ij}D(e_{jk})$. By \eqref{eq:13}, equating the coefficient of $e_{ik}$ from both sides,
\begin{equation}
\label{eq:17a}
a_{ik}^{(ik)}=a_{ij}^{(ij)}+a_{jk}^{(jk)}.
\end{equation}

For $i\neq j$, from $e_{ii}=e_{ij}e_{ji}$, we have
$D(e_{ii})=D(e_{ij})e_{ji}+e_{ij}D(e_{ji})$. By \eqref{eq:9} and \eqref{eq:13}, equating the coefficient of $e_{ii}$ from both sides,
\begin{equation}
\label{eq:17b}
a_{ij}^{(ij)}+a_{ji}^{(ji)}=0.
\end{equation}

Let $X=\sum_{k=1}^{n}\sum_{l=1}^{n}x_{kl}e_{kl}$, where $x_{kl} \in C$. Now, by using \eqref{eq:9}, \eqref{eq:10}, \eqref{eq:13}, \eqref{eq:17a} and \eqref{eq:17b}, $D(X)=BX-XB$, where
\begin{align*}
B=\sum_{l=2}^{n}(-a_{1l}^{(1l)})e_{ll}+ \mathop {\sum_{i=1}^{n}\sum_{j=1}^{n}}_{i\neq j}a_{ij}^{(jj)}e_{ij}.
\end{align*}

Thus, $D$ is an inner derivation.
\end{proof}

As a consequence of Theorem \ref{4.1} and \ref{4.2}, we have the following.

\begin{cor}
\label{4.3}
Every Jordan derivation over $\mathcal{M}_n(C),~n\geq 2$ is an inner derivation.
\end{cor}

Let $D$ be a Jordan derivation on $\mathcal{M}_n(C)$. The question is whether there exists a unique $B\in \mathcal{M}_n(C)$ such that $D(X)=BX-XB$, for all $X\in \mathcal{M}_n(C)$. The answer is given by the following example.
\begin{ex}
Define $D:\mathcal{M}_4(\mathbb{Z}) \rightarrow \mathcal{M}_4(\mathbb{Z})$ by \\
\[D \begin{pmatrix}
x_{11}&x_{12}&x_{13}&x_{14}\\
x_{21}&x_{22}&x_{23}&x_{24}\\
x_{31}&x_{32}&x_{33}&x_{34}\\
x_{41}&x_{42}&x_{43}&x_{44}
\end{pmatrix}= \begin{pmatrix}
0&x_{12}&-x_{12}+x_{13}&x_{14}\\
-x_{21}+x_{31}&x_{32}&-x_{22}+x_{33}&x_{34}\\
-x_{31}&0&-x_{32}&0\\
-x_{41}&0&-x_{42}&0
\end{pmatrix},\\
~\text{where}~x_{ij}\in \mathbb{Z}. \]
Then it is easy to see that $D$ is a Jordan derivation. By Corollary \ref{4.3}, $D$ is an inner derivation. Moreover, $D(X)=BX-XB$, for $B=e_{11}+e_{23}$ or $B= 2e_{11}+e_{22}+e_{33}+e_{44}+e_{23}$. Therefore, for $B$, we have multiple choices.
\end{ex}

\section*{Acknowledgement}
The authors are thankful to DST, Govt. of India for financial support and Indian Institute of Technology Patna for providing the research facilities.

\end{document}